\documentclass[reqno, 11pt]{article}
\usepackage[all]{xy}
\usepackage{latexsym}
\usepackage{amsmath}
\usepackage{ucs}
\usepackage[utf8]{inputenc}
\usepackage[T1]{fontenc}
\usepackage[francais]{babel}
\usepackage{amssymb,amscd,amsthm}
\usepackage[mathscr]{eucal}
\usepackage{bbm}
\usepackage{color}

\newcommand{\p}{\mathfrak{p}}
\newcommand{\cO}{\mathcal{O}}

\newcommand{\F}{\mathbb F}
\newcommand{\N}{\mathbb N}
\newcommand{\Q}{\mathbb Q}

\newcommand{\bP}{\mathbb{P}}
\renewcommand{\a}{\mathfrak{a}}
\renewcommand{\b}{\mathfrak{b}}
\newcommand{\g}{\mathfrak{g}}
\newcommand{\m}{\mathfrak{m}}

\newcommand{\gr}{\mathrm{gr}}

\newcommand{\Ker}{\mathrm{Ker}}
\newcommand{\Coker}{\mathrm{Coker}}

\newcommand{\ra}{\rightarrow}

\newcommand{\End}{\mathrm{End}}

\renewcommand{\hom}{\mathrm{Hom}}

\newcommand{\GL}{\mathrm{GL}}

\newcommand{\SL}{\mathrm{SL}}

\newcommand{\defeq}{\stackrel{\textrm{\tiny{def}}}{=}}
\newcommand{\sym}{\mathrm{Sym}}

\newcommand{\plim}{\varprojlim}

\newcommand{\matr}[4]{\begin{pmatrix}{#1}&{#2}\\
{#3}&{#4}\end{pmatrix}}
\newcommand{\smatr}[4]{\bigl(\begin{smallmatrix} {#1}& {#2}\\
{#3}&{#4}\end{smallmatrix}\bigl)}

\newtheorem{thm}{Th\'eor\`eme}[section]
\newtheorem{lem}[thm]{Lemme}
\newtheorem{cor}[thm]{Corollaire}
\newtheorem{prop}[thm]{Proposition}

\newtheorem{defn}[thm]{D\'efinition}

\newtheorem{rem}[thm]{Remarque}

\begin{document}

\title{Sur la fid\'elit\'e de certaines repr\'esentations de $\GL_2(F)$ sous une alg\`ebre d'Iwasawa}

\author{Yongquan Hu\footnote{Universit\'e de Rennes 1} \and Stefano Morra\footnote{Universit\'e de Versailles-Saint Quentin en Yvelines} \and Benjamin Schraen\footnote{Universit\'e de Versailles-Saint Quentin en Yvelines/CNRS}} \date{}
\maketitle
\begin{abstract}
Soit $F$ une extension finie de $\Q_p$, d'anneau des entiers $\cO_F$ et $E$ une extension finie de $\F_p$. L'action naturelle de $\cO_F^{\times}$ sur $\cO_F$ se prolonge alors en une action continue sur l'alg\`ebre d'Iwasawa $E[[\cO_F]]$.
Dans ce travail, on d\'emontre que les id\'eaux non nuls de $E[[\cO_F]]$ stables par $\cO_F^{\times}$ sont ouverts. En particulier, on en d\'eduit la fid\'elit\'e de l'action de l'alg\`ebre d'Iwasawa des matrices unipotentes sup\'erieures de $\mathrm{GL}_2(\cO_F)$ sur une repr\'esentation lisse irr\'eductible admissible de $\mathrm{GL}_2(F)$.
\end{abstract}
\tableofcontents

\section{Introduction}

Soient $p$ un nombre premier, $F$ une extension finie de $\Q_p$ et $E$ une extension de $\F_p$. Posons $G=\GL_2(F)$. Lorsque $F=\Q_p$, les travaux de Barthel-Livn\'e puis Breuil (\cite{BL}, \cite{BreuilGL2}) aboutissent \`a une classification des $E$-repr\'esentations lisses irr\'eductibles de $G$ ayant un caract\`ere central. Lorsque $F \neq \Q_p$, le m\^eme probl\`eme n'est pas r\'esolu. Malgr\'e des avanc\'es de Breuil et Pa\v{s}k\={u}nas (\cite{BreuilPaskunas}) permettant de construire des familles de repr\'esentations irr\'eductibles, on est encore loin de comprendre quels param\`etres doivent intervenir dans la classification de ces repr\'esentations irr\'eductibles \`a isomorphisme pr\`es (voir par exemple les travaux de l'un d'entre nous dans \cite{Hu}).

Dans ce travail, nous nous int\'eressons \`a un probl\`eme plus simple. Consid\'erons $U$ le sous-groupe compact de $G$ constitu\'e des matrices unipotentes sup\'erieures \`a coefficients dans $\cO_F$ l'anneau des entiers de $F$. Si $\pi$ est une repr\'esentation lisse de $G$ sur un $E$-espace vectoriel, l'action de $U$ sur $\pi$ s'\'etend naturellement en une action de l'alg\`ebre de groupe compl\'et\'ee, ou alg\`ebre d'Iwasawa, $E[[U]]$. Tout \'el\'ement de $\pi$ est annul\'e par un id\'eal ouvert de $E[[U]]$, par lissit\'e de $\pi$. Il est facile de voir qu'un tel id\'eal ne peut annuler toute la repr\'esentation $\pi$ lorsque $\pi$ est de dimension infinie. Cependant, il n'est pas clair \emph{a priori} qu'il n'existe pas d'id\'eal non nul de $E[[U]]$ annulant toute la repr\'esentation $\pi$. En d'autres termes, est-ce que $\pi$ est un module fid\`ele sous l'action de $E[[U]]$ ? Lorsque $F=\Q_p$, la question ne se pose pas car $E[[U]] \simeq E[[X]]$ est un anneau complet de valuation discr\`ete, ses id\'eaux non nuls sont donc tous ouverts. En revanche, $E[[U]]$ est, dans le cas g\'en\'eral, un anneau local r\'egulier de dimension $[F: \Q_p]$. Nous prouvons dans ce travail que si $\pi$ est une repr\'esentation lisse irr\'eductible de dimension infinie de $G$, l'action de $E[[U]]$ est toujours fid\`ele. Plus pr\'ecis\'ement nous prouvons le r\'esultat suivant. Soit $\Gamma=\matr{1+p\cO_F}001 \subset G$. Ce groupe agit contin\^ument sur $U$ par conjugaison, donc sur l'alg\`ebre compl\'et\'ee $E[[U]]$.
\begin{thm}
Soit $I \subset E[[U]]$ un id\'eal stable sous l'action d'un sous-groupe ouvert de $\Gamma$. Alors $I$ est de hauteur $0$ ou $[F:\Q_p]$. Autrement dit, $I=0$ ou $I$ est ouvert dans $E[[U]]$.
\end{thm}

Les techniques que nous employons pour d\'emontrer ce r\'esultat sont classiques pour l'\'etude des alg\`ebres d'Iwasawa (\cite{AWZ}, \cite{AWZ2}, \cite{AWinv}). Le cas o\`u $F$ est non ramifi\'e sur $\Q_p$ peut m\^eme se traiter directement \`a partir du r\'esultat principal de \cite{AWinv}. N\'eanmoins le cas totalement ramifi\'e demande un traitement relativement diff\'erent. Nous avons donc d\^u d\'evisser diff\'eremment les arguments de \cite{AWinv} pour les int\'egrer \`a notre preuve et ainsi traiter de front les diff\'erents cas.\\

\emph{Remerciements : }Nous remercions Christophe Breuil pour avoir port\'e ce probl\`eme \`a notre attention, ainsi que pour plusieurs remarques sur une premi\`ere version de ce travail.\\

\emph{Notations : } Fixons  $p$ un nombre premier et $E$ un corps de caract\'eristique $p$.
Soient  $F$   une extension finie  de $\Q_p$ de degr\'e $n$, $\cO_F$ son anneau des entiers et $\F_q=\F_{p^f}$  son corps r\'esiduel.
Notons $e$ l'indice de ramification de $F$ sur $\Q_p$ de sorte que $n=ef$.

On note $G=\GL_2(F)$, et on d\'efinit deux sous-groupes de $G$ par \[U= \matr1{\cO_F}01\ \  \mathrm{et}\ \  \Gamma=\matr{1+p\cO_F}001.\]
Le groupe $U$ est uniforme, et $\Gamma$ l'est si $p \geq 3$. Si $p=2$, $\Gamma^2$ est uniforme. Le groupe $\Gamma$ est un sous-groupe du normalisateur de $U$ dans $G$, il agit donc contin\^ument sur $U$ par conjugaison. Notons $E[[U]]$ l'alg\`ebre d'Iwasawa de $U$ \`a coefficients dans $E$:
\[ E[[U]]:=\plim_{U'} E[U/U']\]
o\`u $U'$ parcourt les sous-groupes ouverts de $U$. Il s'agit d'un anneau local complet et noeth\'erien, dont l'id\'eal maximal est engendr\'e par les \'el\'ements $u-1$ pour $u\in U$ (cf. \cite[\S$7.4$]{DDMS}).

Comme $U/U^p \simeq \F_q[X]/(X^e)$, le choix d'une uniformisante $\varpi\in \cO_F$ et d'un \'el\'ement primitif $\lambda$ de $\F_q$ sur $\F_p$ nous donne un isomorphisme d'anneaux locaux complets
\begin{eqnarray}
\label{equation-X_ik}
E[[X_{i,k}; 0\leq i\leq e-1, 0\leq k\leq f-1]]\stackrel{\sim}{\longrightarrow} E[[U]]&&
\end{eqnarray}
au moyen de l'identification
\begin{eqnarray*}
X_{i,k}\longmapsto \matr1{\varpi^i[\lambda^k]}01-1
\end{eqnarray*}
pour tout $0\leq i\leq e-1$, $0\leq k\leq f-1$, o\`u $[\lambda^k]\in\cO_F$ d\'esigne le repr\'esentant multiplicatif de $\lambda^k$.

Si $\a$ et $\b$ sont deux id\'eaux d'un anneau commutatif $A$, on note $(\a:\b)$ l'id\'eal des $x \in A$ tels que $x \b \subset \a$.

Si $V$ est un $\F_p$-espace vectoriel on d\'esigne par $\bP(V)$ le projectivis\'e de $V$, c'est-\`a-dire l'ensemble des sous-espaces de dimension $1$ de $V$.
Si $S \subset \bP(V)$, on note alors $\prod_{l \in S} l$ le produit $\prod_{l \in S} w_l \in \mathrm{Sym}(V)$ o\`u $w_l$ est un g\'en\'erateur de la droite $l$ pour tout $l\in S$. Il s'agit d'un \'el\'ement bien d\'efini de $\bP(\mathrm{Sym}(V))$.

Dans toute la suite, lorsque $W$ est un $\F_p$-espace vectoriel, on note $W^*$ l'espace des formes $\F_p$-lin\'eaires de $W$ dans $\F_p$.

\section{Un peu d'alg\`ebre commutative}

Cette partie contient quelques pr\'eliminaires techniques \`a la preuve du th\'eor\`eme principal.

\subsection{La fonction $\delta$}
\label{sec:delta}
Soient $A=E[[X_1,...,X_n]]$ l'anneau  des s\'eries formelles en $n$ variables \`a coefficients dans $E$ et $\m$ son id\'eal maximal. L'anneau $A$ est complet pour la topologie $\m$-adique. On note $\deg$ la fonction degr\'e  associ\'ee \`a la filtration $\m$-adique. Plus pr\'ecis\'ement, pour $x\in A$, $\deg(x)$ est le plus grand entier $k$ tel que $x\in\m^{k}$. Par convention, on pose $\deg(0)=+\infty$. La fonction $\deg$ d\'efinit une valuation sur $A$, autrement dit, on a $\deg(x+y)\geq \min\{\deg(x),\deg(y)\}$ et $\deg(xy)=\deg(x)+\deg(y)$ pour $x,y\in A$. Si $r \in [0,+\infty[$, on notera dans la suite $\m^r$ l'ensemble des $x \in A$ tels que $\deg(x) \geq r$.

Soit $\p$ un id\'eal premier de $A$. Pour $x\in A$, on note $\nu(x)$ le degr\'e de l'image de $x$ dans $A/ \p$ pour la topologie $\m$-adique sur $A/ \p$. Il s'agit du plus grand entier $k$ tel que $x\in \p+\m^k$ avec la convention $\nu(x)=+\infty$ si $x\in \bigcap_{k\geq 1}(\p+\m^k)$. Comme $A$ est noeth\'erien et complet pour la topologie $\m$-adique, tout ses id\'eaux sont ferm\'es dans $A$, donc $\p=\bigcap_{k\geq 1}(\p+\m^k)$. Ainsi $\nu(x)=+ \infty$ si et seulement si $x \in \p$. On a toujours $\nu(x)\geq \deg(x)$, mais $\nu$ n'est pas n\'ecessairement une valuation. On a seulement la propri\'et\'e plus faible $\nu(xy)\geq \nu(x)+\nu(y)$ si $x,y\in A$.

\begin{defn}
Pour $x, P\in A$, on d\'efinit:
\[\delta(x,P):=\sup_{k\geq 0}\bigl\{\nu(x^kP)-\deg(x^kP)\bigr\}\in \N\cup\{+\infty\}\]
avec la convention $\delta(x,P)=+\infty$ si $x\in\p$ ou $P\in\p$.  On \'ecrira $\delta(x)$ au lieu de $\delta(x,1)$.
\end{defn}

Comme
\[ \nu(x^{k+1}P)-\deg(x^{k+1}P) \geq \nu(x^k P)-\deg(x^k P)+\nu(x)- \deg(x), \]
la fonction $k \mapsto \nu(x^k P)-\deg(x^k P)$ est croissante, et donc
\[\delta(x,P)=\underset{k\rightarrow \infty}{\lim}\nu(x^kP)-\deg(x^kP)\]
pour tout $x, P\in A$.

\vspace{2mm}

Soit $\gr(A)=\bigoplus_{k\geq 0}\m^k/\m^{k+1}$ l'anneau gradu\'e associ\'e \`a $A$. Il est canoniquement isomorphe \`a l'anneau des polyn\^omes $E[X_1,...,X_n]$. 
Pour $x\in A$ non nul, on note $\gr(x):=x \mod \m^{\deg(x)+1}\in \gr(A)$ le symbole de $x$. On note $\gr(\p)$ l'id\'eal gradu\'e associ\'e \`a $\p$, c'est-\`a-dire,
\[ \gr(\p)=\{ \gr(x),Â \, x \in \p \}.\]
C'est un id\'eal gradu\'e de $\gr(A)$ et
\[\gr(\p):=\bigoplus_{k\geq 0}(\p\cap \m^k)/(\p \cap \m^{k+1})\subseteq \gr(A)\]
Le lemme suivant r\'esume quelques propri\'et\'es de la fonction $\delta$.

\begin{lem}\label{lem-delta}
Soit $x\in A$ non nul.
\begin{enumerate}
\item[(i)] La fonction $\delta$ ne prend que les valeurs $0$ et $+ \infty$ et $\delta(x)=+\infty \Leftrightarrow \gr(x)\in \sqrt{\gr(\p)}$.

\item[(ii)] On a $\delta(x)=0$ (resp. $=+\infty$) si et seulement si pour tout $P\notin \p$, $\delta(x,P)<+\infty$ (resp. $=+\infty$).

\item[(iii)] Supposons $x\neq 0$ et $\delta(x)=+\infty$. Il existe  $\epsilon>0$ et $k_0 \geq 0$ tels que pour $k\geq k_0$,
\[\nu(x^k)\geq (1+\epsilon) \deg(x^k).\]
\end{enumerate}
\end{lem}
\begin{proof}
(i) La condition $\gr(x)\in\sqrt{\gr(\p)}$ est \'equivalente    \`a l'existence de $k_0\geq 1$ tel que
\[\nu(x^{k_0})\geq \deg(x^{k_0})+1.\]
Par cons\'equent, si $\gr(x)\in\sqrt{\gr(\p)}$, alors
\[\nu(x^{m{k_0}})\geq m\nu(x^{k_0})\geq m+ \deg(x^{mk_0}), \]
d'o\`u $\delta(x)=+\infty$.  D'autre part, si $\gr(x)\notin\sqrt{\gr(\p)}$, alors $\nu(x^k)=\deg(x^k)$ pour tout $k\geq 1$, donc $\nu(x)=0$.

(ii) Montrons tout d'abord que $\delta(x)=+\infty$ si et seulement si $\delta(x,P)=+\infty$ pour tout $P \in A$. En prenant $P=1$, on voit que la condition est suffisante. De plus, comme $\nu(x^kP)\geq \nu(x^k)+\deg(P)$, on a bien $\delta(x,P)=+\infty$ si $\delta(x)=+\infty$, d'o\`u la n\'ecessit\'e.

Supposons maintenant $\delta(x)=0$. D'apr\`es \cite[corollaire 3.14]{Na} (qui est une cons\'equence du lemme d'Artin-Rees), pour $P \notin \p$, il existe un entier $k_0\geq 0$ tel que:
\[((\p+\m^{k}):P)\subseteq (\p:P)+\m^{k-k_0}=\p+\m^{k-k_0},\ \ \forall k\geq k_0 \]
o\`u l'\'egalit\'e vient de l'hypoth\`ese $P\notin\p$.
On en d\'eduit que pour $k$ assez grand (tel que $\nu(x^kP)\geq k_0$):
\[\nu(x^k)\geq \nu(x^kP)-k_0,\]
puis
\[\nu(x^kP)-\deg(x^kP) \leq (\nu(x^k)-\deg(x^k))+(k_0-\deg(P))=k_0-\deg(P).\] D'o\`u le r\'esultat. La suffisance s'obtient en consid\'erant le cas o\`u $P=1$.

(iii)
La preuve de (i) nous donne un entier $k_0\geq 1$ tel que $\nu(x^{k_0})\geq 1+\deg(x^{k_0})$. Notons que la condition $\delta(x)=+\infty$ entra\^ine $\deg(x)\neq0$.
Pour $k\geq k_0$ avec $m=\lfloor k/k_0\rfloor$ la partie enti\`ere de $k/k_0$, on a 
\[\nu(x^k)\geq m+\deg(x^k)\geq \deg(x^{k})\bigl(1+\frac{m}{k\deg(x)}\bigr)\geq \deg(x^k)\bigl(1+\frac{1}{2k_0\deg(x)}\bigr).
\]
On prend donc $\epsilon= \frac{1}{2k_0\deg(x)}$.
 \end{proof}

\subsection{Le contr\^oleur d'un id\'eal}

Rappelons que $U=\smatr1{\cO_F}01$ et $E[[U]]$  est l'alg\`ebre d'Iwasawa de $U$. Soient $I$ un id\'eal de $E[[U]]$ et $V$ un sous-groupe ferm\'e de $U$. On dit que $I$ est contr\^ol\'e par $V$, ou que $V$ contr\^ole $I$, si $I$ est engendr\'e topologiquement par un sous-ensemble de $E[[V]]$ ou, de mani\`ere \'equivalente, si
\[I=\overline{(I\cap E[[V]])\cdot E[[U]]}.\]

Rappelons un crit\`ere pour que $I$ soit control\'e par le sous-groupe ouvert $\varpi U$.
\begin{lem}\label{lem-controle}
L'id\'eal $I$ est contr\^ol\'e par $\varpi U$ si et seulement si $I$ est stable par les d\'erivations $\frac{\partial}{\partial X_{0,k}}$ pour $0\leq k\leq f-1$.
\end{lem}
\begin{proof}
C'est un cas particulier de \cite[proposition 2.4(d)]{AWZ2}.
\end{proof}

\subsection{Quelques r\'esultats sur les matrices de polyn\^omes}

Pour pouvoir appliquer le lemme \ref{lem-controle}, il nous faut pouvoir \'ecrire les d\'erivations $\frac{\partial}{\partial X_{0,k}}$ en fonctions des d\'erivations de $E[[U]]$ induites par l'action de $\Gamma$. C'est l'objet de cette partie. Plus pr\'ecis\'ement consid\'erons $V$ un $\F_p$-espace vectoriel de dimension finie. Les d\'erivations de l'alg\`ebre $\sym(V)$ s'identifient aux morphismes $\sym(V)$-lin\'eaires $\hom_{\sym(V)}(\sym(V) \otimes_{\F_p} V, \sym(V))$, c'est-\`a-dire aux applications $\F_p$-lin\'eaires de $V$ dans $\sym(V)$. Nous allons d\'esormais raisonner en termes d'applications $\F_p$-lin\'eaires. Soit $\phi$ l'inclusion canonique de $V$ dans $\sym(V)$. On note $\phi^{p^n}$ l'application $\F_p$-lin\'eaire de $V$ dans $\sym(V)$ obtenue en posant $\phi^{p^n}(v)=\phi(v)^{p^n}$ pour tout $v \in V$. La proposition $1.4$ de \cite{AWZ} montre que si $g$ est une forme lin\'eaire de $V$, alors, quitte \`a multiplier $g$ par un \'el\'ement homog\`ene bien choisi de $\sym(V)$, on peut \'ecrire $g$ comme une combinaison $\sym(V)$-lin\'eaires des applications $\phi^{p^n}$ pour $n_0 \leq n \leq n_0+\dim(V)-1$. Dans cette partie, nous reprenons la preuve de \cite[proposition 1.4]{AWZ} afin d'obtenir une borne inf\'erieure pour les degr\'es des coefficients de cette combinaison lin\'eaire. 

\begin{prop} \label{prop-Uf-pre}
Fixons $g\in V^{*}$. Soit $m=\dim V$. Alors pour $s\geq 0$, on a l'\'egalit\'e suivante dans $\hom_{\F_p}(V, \sym(V))$, o\`u $\m_V$ d\'esigne l'id\'eal maximal de $\sym(V)$ engendr\'e par $V$,
\[\Bigl(\prod_{w\in \bP(V)\backslash \bP(\ker g)}w\Bigr)^{p^s}\cdot  g \in \sum_{j=1}^m\m_V^{p^s(p^{m-1}-p^{j-1})}\phi^{p^{s+j-1}}.\]
\end{prop}

La preuve de la proposition \ref{prop-Uf-pre} consiste essentiellement \`a combiner la r\`egle de Cramer \`a des estimations du degr\'e des mineurs d'une matrice de Vandermonde. Les notations suivantes sont issues de \cite[\S1]{AWZ}.

\vspace{2mm}

Soit $\{w_1,...,w_m\}$ une base de $V$ sur $\F_p$.  On note $M(w_1,...,w_m)\in M_{m}(\sym(V))$ la matrice de type Vandermonde:
\begin{equation}\label{equation-matrix-Vander}M(w_1,...,w_m)=\left(\begin{array}{cccc}w_1&w_2&\cdots&w_m\\
w_1^p&w_2^p&\cdots&w_m^p\\
\vdots&\vdots&\ddots&\vdots\\
w_1^{p^{m-1}}&w_2^{p^{m-1}}&\cdots&w_m^{p^{m-1}}\end{array}\right)\end{equation}
Soit $\mathrm{Com}(M(w_1,...,w_m))$ sa comatrice, dont le $(i,j)$-terme est \'egal \`a $(-1)^{i+j}\det C_{ji}$, $C_{ij}$ \'etant le bloc de $M(w_1,...,w_m)$ obtenu en enlevant la $i$-\`eme ligne et la $j$-\`eme colonne. La r\`egle de Cramer s'\'ecrit \[M(w_1,...,w_m)\cdot \mathrm{Com}(M(w_1,...,w_m))=\det M(w_1,..,w_m)\cdot \mathrm{Id}_m.\]

\vspace{2mm}

Pour $j\in\{1,...,m\}$, posons $W_j\subset V$ le sous-$\F_p$-espace vectoriel engendr\'e par $\{w_i: i\neq j\}$ et   d\'efinissons les matrices
\[M(w_1,...,\hat{w}_j,...,w_m):=M(w_1,...,w_{j-1},w_{j+1},...,w_m).\]
Il est prouv\'e dans \cite[proposition 1.2]{AWZ} que $\det M(w_1,...,\hat{w}_j,...,w_m)$ divise tous les $\det C_{ij}$ pour $1\leq i\leq m$ et que $\det(M(w_1, \dots, w_m))$ est un polyn\^ome homog\`ene de degr\'e $|\mathbb{P}(\F_p^m)|$ (\cite[Lemma $1.1$]{AWZ}).
\begin{lem}\label{lem-estimation}
Pour $1\leq i\leq m$, on a
\[\frac{\det C_{ij}}{\det M(w_1,...,\hat{w}_j,...,w_m)}\in \m_V^{p^{m-1}-p^{i-1}}.\]
\end{lem}
\begin{proof}
Il suffit de remarquer que $\det C_{ij}$ est un polyn\^ome de degr\'e sup\'erieur ou \'egal \`a
\[(1+p+\cdots+p^{m-1})-p^{i-1},\]
et $\det M(w_1,...,\hat{w}_j,...,w_m)$ est homog\`ene de degr\'e $1+p+\cdots+p^{m-2}$.
\end{proof}

\begin{proof}[D\'emonstration de la proposition \ref{prop-Uf-pre}]
On peut supposer que $g \neq 0$, car sinon l'\'enonc\'e est \'evident. Posons $f_1:=g$ et compl\'etons  $\{f_1\}$ en une base
$\{f_1,\dots,f_m\}$ de $V^{*}$. Soit $\{w_1,\dots,w_{m}\}\subset V$ la base duale, de telle sorte que $\phi=\sum_{j=1}^m w_j  f_j$.
Par construction on a alors
\begin{eqnarray}
\label{equa-phi}
\phi^{p^r}=\sum_{j=1}^mw_j^{p^r}f_j
\end{eqnarray}
pour tout entier $r\geq 0$. Soient $\mathbf{e},\mathbf{f}\in \mathrm{Hom}_{\F_p}(V, \sym(V))^{m}$ les vecteurs colonnes d\'efinis par 
\[\mathbf{e}=\left(\begin{array}{c }\phi^{p^s}  \\ \vdots   \\ \phi^{p^{s+m-1}} \end{array}\right),\ \ \mathbf{f}=\left(\begin{array}{c }f_1  \\ \vdots   \\ f_m \end{array}\right).\]
Les \'equations (\ref{equa-phi}) pour $r=s,\,s+1,\dots,s+m-1$ s'\'ecrivent matriciellement 
\begin{eqnarray}
\label{equa-comp}
M(w_1^{p^s},\dots,w_m^{p^s})\cdot\mathbf{f}=\mathbf{e}.
\end{eqnarray}

Posons maintenant $\Delta_j= \prod_{w\in \bP(V)\backslash \bP(\ker(f_j))} w$. Par \cite[lemma 1.1(2)]{AWZ}, on a
\[ \Delta_j^{p^s}=\lambda_j\cdot \frac{\det M(w_1^{p^s},\dots,w_m^{p^s})}{\det M(w_1^{p^s},\dots,\hat{w}_j^{p^s},\dots, w_m^{p^s})} \]
avec $\lambda_j\in \F_p^{\times}$. En d\'esignant par $H$  la matrice diagonale dont le $(j,j)$-\`eme terme est  $\lambda_j^{-1}\cdot\det M(w_1^{p^s},\dots,\hat{w}^{p^s}_j,\dots, w_m^{p^s})$ et par  $D$ la matrice diagonale dont le
$(j,j)$-\`eme terme est $\Delta_j$, on trouve:
\begin{eqnarray}
\label{equa-2}
\underbrace{H^{-1}\cdot\mathrm{Com}(M(w_1^{p^s},...,w_m^{p^s}))}_{\defeq U}\cdot M(w_1^{p^s},...,w_m^{p^s})= D^{p^s}
\end{eqnarray} 
et,  d'apr\`es le lemme \ref{lem-estimation}, $U$ est une matrice dont le $(j,i)$-\`eme terme est un \'el\'ement de $\m_V^{p^s(p^{m-1}-p^{i-1})}$.

La premi\`ere ligne de l'\'egalit\'e de matrices $U\cdot \mathbf{e}=D^{p^s}\cdot\mathbf{f}$, d\'eduite de (\ref{equa-comp}) et (\ref{equa-2}), nous permet donc de conclure.
\end{proof}

Soit $A$ la compl\'etion de $\sym(V)$ pour la topologie d\'efinie par l'id\'eal maximal de $\sym(V)$ engendr\'e par $V$. Soit $\m$ l'id\'eal maximal de $A$. Consid\'erons \'egalement la situation plus g\'en\'erale o\`u $V_1$ est un $\F_p$-espace vectoriel de dimension finie et $\varphi \in \hom_{\F_p}(V_1,V)$, d'image $V_2 \subset V$. On note encore par la lettre $\varphi$ l'application $\F_p$-lin\'eaire de $V_1$ dans $\sym(V)$ obtenue par composition avec $V \subset \sym(V)$. Il est plus agr\'eable pour la suite de reformuler la proposition \ref{prop-Uf-pre} sous la forme suivante.

\begin{prop}\label{prop-Uf}
Fixons $g\in V_2^{*}$. Soit $m=\dim V_2$. Alors pour $s\geq 0$, on a, dans $\hom_{\F_p}(V_1, A)$,
\begin{equation} \label{eq-prop-Uf} \Bigl(\prod_{w\in \bP(V_2)\backslash \bP(\ker g)}w\Bigr)^{p^s}\cdot  (g\circ\varphi) \in \sum_{j=1}^m\m^{p^s(p^{m-1}-p^{j-1})}\varphi^{p^{s+j-1}}.
\end{equation}
\end{prop}

\begin{proof}
Il suffit de composer le r\'esultat de la proposition \ref{prop-Uf-pre}, appliqu\'e \`a $V_2$, \`a droite avec $\varphi$, et \`a gauche avec l'inclusion $\sym(V_2) \subset \sym(V) \subset A$.
\end{proof}

\section{Le r\'esultat principal}

Cette partie est consacr\'ee \`a la d\'emonstration du th\'eor\`eme principal.

\begin{thm}\label{thm-principal}
Soient $\Gamma'$ un sous-groupe ouvert de $\Gamma$ et $I$ un id\'eal non nul de $E[[U]]$ stable par $\Gamma'$. Alors $I$ est un id\'eal ouvert de $E[[U]]$.
\end{thm}
\medskip

Commen\c{c}ons par nous ramener au cas o\`u $I$ est premier.

\begin{lem}\label{lem-reduction}
Si l'id\'eal maximal $\m$ de $E[[U]]$ est le seul id\'eal premier non nul de $E[[U]]$ stable par un sous-groupe ouvert de $\Gamma$, alors le th\'eor\`eme \ref{thm-principal} est vrai.
\end{lem}
\begin{proof} L'anneau $E[[U]]$ \'etant noeth\'erien, l'ensemble $\mathrm{Ass}(E[[U]]/I)$ des id\'eaux premiers associ\'es \`a $I$ est fini, ses \'el\'ements sont les id\'eaux premiers de la forme
\[(I:s),\ \ s\in E[[U]].\]
On v\'erifie directement que si $\gamma\in \Gamma'$ et si $(I:s)$ est premier, alors $(I:\gamma(s))$ l'est aussi. Autrement dit, l'ensemble $\mathrm{Ass}(E[[U]]/I)$ est stable par $\Gamma'$. Comme $\mathrm{Ass}(E[[U]]/I)$ est fini, il existe un sous-groupe ouvert $\Gamma''$ de $\Gamma'$ qui fixe tous les \'el\'ements de $\mathrm{Ass}(E[[U]]/I)$. On peut donc supposer que $I$ est premier, car $I$ est ouvert si et seulement si tous les \'el\'ements de $\mathrm{Ass}(E[[U]]/I)$ le sont.
\end{proof}
\medskip

Soit $\mathfrak{m}$ l'id\'eal maximal de $E[[U]]$ et  d\'efinissons
\[V:= \bigoplus_{0\leq i\leq e-1,0\leq k\leq f-1}\F_p X_{i,k}.\]
D'apr\`es l'isomorphisme (\ref{equation-X_ik}), on a $
V \otimes_{\F_p} E \simeq \m / \m^2$.
On peut ainsi identifier $V$ \`a un $\F_p$-sous-espace de $\m/\m^2$, ce que nous ferons dor\'enavant sans aucun commentaire ult\'erieur. Posons enfin $Y_i=\bigoplus_{k=0}^{f-1}\F_p X_{i,k}$ de sorte que $V=\bigoplus_{i=0}^{ e-1} Y_i$.

\begin{defn}
On d\'efinit un morphisme $\rho:\cO_F/p\cO_F\ra \End_{\F_p}(V)$ comme suit: si $\overline{x}\in\cO_F/p\cO_F$ et si $x\in\cO_F$ est un rel\`evement de $\overline{x}$, alors
\[\rho(\overline{x})(X_{i,k}):=\matr1{x\cdot\varpi^i[\lambda^k]}01 -1 \mod \m^2.\]
\end{defn}
\begin{rem} Si on pose $\g=\cO_F/p \cO_F$, il s'agit de l'action de $\g$ sur $V$ d\'ecrite dans \emph{\cite[\S$4.2$]{AWinv}}. Comme notre cas est tr\`es particulier, nous n'utilisons pas le formalisme de cet article.
\end{rem}
Puisque $\smatr1{x+y}01-1=\smatr1{x}01\smatr1{y}01 -1\equiv \bigl(\smatr1x01-1\bigr)+\bigl(\smatr1y01-1\bigr)\mod \m^2,$
et que $\smatr1{py}01-1=\bigl(\smatr1y01-1\bigr)^p\in\m^p$ pour tout $x,y\in\cO_F$, on voit que la d\'efinition ci-dessus ne d\'epend pas du choix de $x$ et que $\rho$ est bien un morphisme d'espaces vectoriels sur $\F_p$. On peut identifier  $\cO_F/p\cO_F$ avec $\F_q[\varpi]/(\varpi^e)=\bigoplus_{i=0}^{e-1}\F_q\varpi^i$.
\begin{lem}\label{lem-rho(x)}
Fixons $i\in\{0,...,e-1\}$.
\begin{enumerate}
\item[(i)] Soit $\overline{x}\in \F_q\varpi^{i}$. Alors $\rho(\overline{x})(Y_j)=\left\{\begin{array}{clc}Y_{i+j}& \mathrm{si}\ i+j\leq e-1 \\
0& \mathrm{si}\ i+j\geq e.\end{array}\right.$

\item[(ii)] Soient $0\leq j\leq e-1-i$ et $g\in Y_{i+j}^*$ non nul. Alors l'application de $\F_q\varpi^i$ dans $Y_{j}^*$ donn\'ee par $\overline{x}\mapsto g\circ\rho(\overline{x})$ est une bijection.
\end{enumerate}
\end{lem}
\begin{proof}
L'application $u \mapsto (u-1)+\m^2$ est un isomorphisme de groupes de $U/U^p$ sur $V$ et l'action de $\rho(\overline{x})$ sur $V$ induit une action sur $U/U^p \simeq \cO_F / p \cO_F$ qui est donn\'ee par la multiplication par $\overline{x}$. Les deux \'enonc\'es s'en d\'eduisent.
\end{proof}
\medskip

\begin{lem}\label{lem-hauteur}
Soit $\p$ un id\'eal premier de $E[[U]]$. Pour que $\p$ soit l'id\'eal maximal de $E[[U]]$, il faut et il suffit que $\sqrt{\gr(\p)}=\langle X_{i,k}; 0\leq i\leq e-1,0\leq k\leq f-1\rangle$ dans $\gr(E[[U]])\cong E[X_{i,k}]$.
\end{lem}
\begin{proof}
La n\'ecessit\'e est imm\'ediate. Prouvons la suffisance. Si $\gr(\p)$ contient une puissance de $\gr(\m)$, on a $\p+\m^{n+1}=\m^n$ pour $n$ assez grand. En utilisant le lemme de Nakayama, on montre qu'alors $\p=\m^n$ puis $\p=\m$ vu que $\p$ est premier.
\end{proof}

Consid\'erons un id\'eal premier $\p$ non nul qui n'est pas maximal. Le lemme \ref{lem-hauteur} implique que $\sqrt{\gr(\p)}$ ne contient pas tous les $\gr(X_{i,k})$. Il existe donc un indice $i_0\in\{0,...,e-1\}$ tel que
\begin{enumerate}
\item[(a)] $\gr(X_{j,k})\in\sqrt{\gr(\p)}$ pour tout $j> i_0$ et tout $0\leq k\leq f-1$;

\item[(b)] il existe $k'\in\{0,...,f-1\}$ tel que $\gr(X_{i_0,k'})\notin \sqrt{\gr(\p)}$.
\end{enumerate}

\begin{prop} \label{prop-d(x)}
Soit $\p$ un id\'eal premier de $E[[U]]$ stabilis\'e par un sous-groupe ouvert $\Gamma' \subset \Gamma$, v\'erifiant les conditions (a) et (b) ci-dessus. Soit $F\in \p$. Il existe $\epsilon>0$ tel que pour tout $r$ suffisamment grand l'on ait
\begin{equation} \label{eq-taylor} \forall \overline{x}\in\F_q\varpi^{i_0}, \ \ \ \sum_{k=0}^{f-1}\rho(\overline{x})(X_{0,k})^{p^r}\cdot (1+X_{0,k})\frac{\partial F}{\partial X_{0,k}}\in \p+\m^{p^{r}(1+\epsilon)}.
\end{equation}
\end{prop}

\begin{proof}
Soit $\overline{x} \in \F_q \varpi^{i_0}$. En choisissant $x\in\cO_F$ un rel\`evement de $\overline{x}$,  on a, par d\'efinition, pour tout $0\leq i\leq e-1, 0\leq k\leq f-1$:
\[\rho(\overline{x})(X_{i,k})^{p^r}\equiv \smatr1{p^rx\cdot\varpi^i[\lambda^k]}01 -1 \mod \m^{2p^r}.\]
Posons $\gamma=\smatr{1+p^rx}001$ avec $r$ suffisamment grand pour que $\gamma\in\Gamma'$.
Par d\'efinition de l'action de $\gamma$ sur $U$, on obtient
\[\gamma(X_{i,k})=\smatr{1}{(1+p^rx)\varpi^i[\lambda^k]}01-1\equiv X_{i,k}+(1+X_{i,k})\cdot \rho(\overline{x})(X_{i,k})^{p^r}\mod \m^{2p^r}.\]
Soit maintenant $F\in \p$. En \'ecrivant le d\'eveloppement de Taylor \`a l'ordre 1 de $\gamma(F)=F(\gamma(X_{i,k}))$, on voit que 
\[ \gamma(F)-
F -\sum_{i,k}\rho(\overline{x})(X_{i,k})^{p^r}\cdot (1+X_{i,k})\frac{\partial F}{\partial X_{i,k}}\in\m^{2p^{r}}.\]
Comme $\p$ est stable par $\Gamma'$, on a $\gamma(F)\in \p$ et donc
\[\sum_{i,k}\rho(\overline{x})(X_{i,k})^{p^r}\cdot (1+X_{i,k})\frac{\partial F}{\partial X_{i,k}}\in \p+\m^{2p^{r}}.\]
Pour $\overline{x}\in \F_q\varpi^{i_0}$, le lemme \ref{lem-rho(x)} implique que
\begin{equation} \label{inclusions} \rho(\overline{x})(Y_0)\subseteq Y_{i_0}, \ \ \dots  , \ \ \rho(\overline{x})(Y_{e-1-i_0})\subseteq Y_{e-1} \ \ \mathrm{et}\ \ \rho(\overline{x})(Y_i)=0\ \  \forall i\geq e-i_0.
\end{equation}
Par ailleurs, l'hypoth\`ese sur $i_0$ et le lemme \ref{lem-delta} $(i)$ montrent que $\delta(X_{j,k})=+\infty$ pour $j>i_0$ et $0 \leq k \leq f-1$. Autrement dit, $\delta(y)=+\infty$ pour $y \in Y_j$, $j  >i_0$. Combin\'ees \`a \eqref{inclusions}, ces \'egalit\'es nous donnent $
\delta(\rho(\overline{x})(X_{i,k}))=+\infty$ pour tout $1\leq i\leq e-1-i_0$, $0\leq k\leq f-1$. 
Comme $\deg (\rho(\overline{x})(X_{i,k})^{p^r})\geq p^r$ on d\'eduit du lemme \ref{lem-delta} $(iii)$ l'existence de $\epsilon>0$ tel que
\[
\rho(\overline{x})(X_{i,k})^{p^r}\cdot (1+X_{i,k})\frac{\partial F}{\partial X_{i,k}}\in \p+\m^{(1+\epsilon)p^r}
\]
pour tout $1\leq i\leq e-1-i_0$, $0\leq k\leq f-1$ et $r$ suffisamment grand.
Ceci permet de conclure.
\end{proof}

Le corollaire qui suit est le r\'esultat cl\'e pour d\'emontrer le th\'eor\`eme principal.
Il utilise l'estimation \'etablie dans la proposition \ref{prop-Uf}.

\begin{cor}
\label{cor-delta}
Conservons les notations de la proposition \ref{prop-d(x)}. Soient $g\in Y_{i_0}^{*}$ une forme lin\'eaire non nulle et $\overline{x}$ un \'el\'ement de $\F_q\varpi^{i_0}\backslash \{0\}$. Posons
\[P=\sum_{k=0}^{f-1} (g\circ\rho(\overline{x}))(X_{0,k})\cdot(1+X_{0,k})\frac{\partial F}{\partial X_{0,k}}.\]
Alors, en notant $U_g=\prod_{w\in\bP(Y_{i_0})\backslash\bP(\ker(g))}w$, on a
$\delta(U_g,P )=+\infty$.
\end{cor}
\begin{proof}
D'apr\`es la d\'efinition de la fonction $\delta$ (cf. \S \ref{sec:delta}), on peut supposer que $U_g,P\notin\p$ et il suffit de montrer que
\[
\sup_{r\geq 0}\bigl\{\nu(U_g^{p^r}P)-\deg(U_g^{p^r}P)\bigr\}=+\infty.
\]
Notons d'abord que $U_g$ est un polyn\^ome homog\`ene de degr\'e $p^{f-1}$, de telle sorte que
\[
\deg(U_g^{p^r}P)= p^{r+f-1}+\deg(P).
\]
De plus, la proposition \ref{prop-Uf} appliqu\'ee \`a $\varphi=\rho(\overline{x})\vert_{Y_0}$, $V_1=Y_0$ et $V_2=Y_{i_0}$, ainsi que le morphisme local canonique de changement de base $A\hookrightarrow E[[U]]$, nous donnent
$$
U_g^{p^r}\cdot (g\circ\rho(\overline{x}))(X_{0,k})\in \sum_{j=1}^{f}\m^{p^r(p^{f-1}-p^{j-1})}\cdot \rho(\overline{x})(X_{0,k})^{p^{r+j-1}}
$$
pour tout $r\geq0$, ce qui nous permet d'\'ecrire, au moyen de  l'inclusion \eqref{eq-taylor} 
 \[\begin{array}{rll}  U_g^{p^r}P
&\in& \displaystyle\sum_{k=0}^{f-1}\Bigl(\sum_{j=1}^{f}\m^{p^r(p^{f-1}-p^{j-1})}\cdot \rho(\overline{x})(X_{0,k})^{p^{r+j-1}} \Bigr)(1+X_{0,k})\frac{\partial F}{\partial X_{0,k}}\\
&\subset&\displaystyle\sum_{j=1}^{f}\m^{p^r(p^{f-1}-p^{j-1})}\cdot (\p+\m^{p^{r+j-1}(1+\epsilon)})\\
&\subset& \p+\m^{p^r(p^{f-1}+\epsilon)}\end{array}\]
pour tout $r$ suffisamment grand.
Cela entra\^ine que
$$
\nu(U_g^{p^r}P)\geq p^{r+f-1}+\epsilon p^r
$$
et donc $\delta(U_g,P)=+\infty$, d\`es que $\epsilon>0$.
\end{proof}

\begin{proof}[D\'emonstration du th\'eor\`eme \ref{thm-principal}]
Soit $\p \subset E[[U]]$ un id\'eal premier stable par un sous-groupe ouvert de $\Gamma$. Supposons par l'absurde que $\p$ n'est ni nul, ni maximal, et  soit $i_0$ l'indice satisfaisant les conditions (a) et (b) comme pr\'ec\'edemment.
Il existe $g\in Y_{i_0}^*$ non nul tel que  $\gr(U_g)\notin\sqrt{\gr(\p)}$. En effet, en appliquant le th\'eor\`eme des z\'eros de Hilbert, on voit que l'id\'eal gradu\'e de $E[X_{i_0,0}, \dots, X_{i_0,f-1}]$ engendr\'e par tous les $U_g$ est ouvert dans $E[X_{i_0,0}, \dots, X_{i_0,f-1}]$. La condition (b) implique donc l'existence d'un tel $g$.

D'apr\`es le lemme \ref{lem-delta} $(i)$, on a alors $\delta(U_g)=0$. Le corollaire \ref{cor-delta} et le lemme \ref{lem-delta}(ii) impliquent donc
\[\sum_{k=0}^{f-1}(g\circ\rho(\overline{x}))(X_{0,k})\cdot (1+X_{0,k})\frac{\partial F}{\partial X_{0,k}}\in\p,\ \ \forall \overline{x}\in\F_q\varpi^{i_0}.\]
Le lemme \ref{lem-rho(x)} assure l'existence de $\overline{x}_k\in \F_q\varpi^{i_0}$ tel que $g\circ\rho(\overline{x}_k)=X_{0,k}^{*}$, o\`u $\{X_{0,k}^*, \, 0 \leq kÂ \leq f-1 \}$ d\'esigne la base duale de $Y_0$. On en d\'eduit que
\[ (1+X_{0,k})\frac{\partial F}{\partial X_{0,k}}\in \p,\ \ \forall 0\leq k\leq f-1,\]
et, puisque $1+X_{0,k}$ est inversible dans $E[[U]]$, que $\frac{\partial F}{\partial X_{0,k}}\in \p$. Ceci \'etant vrai pour tout $0 \leq k \leq f-1$, le lemme \ref{lem-controle} implique que $\p$ est contr\^ol\'e par $U_1=\varpi U$.\\

Comme $E[[U]]$ est une $E[[U_1]]$-alg\`ebre finie, l'id\'eal $\p \cap E[[U_1]]$ est un id\'eal premier de $E[[U_1]]$ qui n'est ni nul ni maximal. La multiplication par $\varpi$ induit un isomorphisme de pro-$p$-groupes $U\cong U_1$ qui est $\Gamma$-\'equivariant. En it\'erant l'argument pr\'ec\'edent, on voit qu'en fait $\p$ est contr\^ol\'e par tous les sous-groupes $\varpi^n U$ pour $n \in \N$. Comme $\bigcap_{n \in \N} \varpi^n U=\{0\}$, on obtient $\p=0$.
\end{proof}

\section{Applications}

\begin{cor}
Soit $(\pi, \rho)$ une repr\'esentation lisse irr\'eductible de $G=\GL_2(F)$ sur un $E$-espace vectoriel de dimension infinie. Alors l'action de $E[[U]]$ sur $\pi$ est fid\`ele.
\end{cor}

\begin{proof}
Soit $I$ le noyau de $E[[U]] \rightarrow \mathrm{End}_E(\pi)$. C'est un id\'eal de $E[[U]]$ stable sous l'action de $\cO_F^{\times}$. 
D'apr\`es le th\'eor\`eme \ref{thm-principal}, $I$ est ouvert ou nul. Supposons qu'il soit ouvert. Alors il existe un sous-groupe ouvert $U' \subset U$ agissant trivialement sur $\pi$. Soit $N \subset G$ le sous-groupe des matrices unipotentes sup\'erieures. Ce groupe $N$ est une union de conjugu\'es de $U'$. Ainsi le noyau de $\rho :Â \, G \rightarrow \mathrm{Aut}(\pi)$ contient $U'$, donc $N$. Or un sous-groupe distingu\'e de $G$ contenant $N$ contient $\SL_2(F)$. Il est alors bien connu que $\pi$,
\'etant irr\'eductible, est de dimension finie.
\end{proof}
\begin{cor}\label{cor-Ass}
Soit $M$ un $E[[U]]$-module de type fini muni d'une action semi-lin\'eaire de $\Gamma$ et soit $\mathrm{Ass}(M)$ l'ensemble des id\'eaux premiers associ\'es \`a $M$ (\cite[\S1]{Bou}). Alors $\mathrm{Ass}(M)\subset\{\{0\},\m\}$. Si de plus $M$ est de torsion, alors $\mathrm{Ass}(M)\subset\{\m\}$ et $M$ est de longueur finie.
\end{cor}
\begin{proof}
Comme dans la preuve du lemme \ref{lem-reduction}, on voit que $\mathrm{Ass}(M)$ est fini et ses \'el\'ements  sont tous stables par un certain sous-groupe ouvert de $\Gamma$.  Le premier \'enonc\'e s'en d\'eduit.  Pour le deuxi\`eme, il suffit de remarquer que $\{0\}\notin\mathrm{Ass}(M)$ lorsque $M$ est de torsion.
\end{proof}

\begin{cor}\label{cor-reflexif}
Soit $M$ un $E[[U]]$-module de type fini muni d'une action semi-lin\'eaire de $\Gamma$. Si on note $\mathrm{T}(M)$ le sous-$E[[U]]$-module de torsion de $M$, $\mathrm{T}(M)$ est de longueur finie sur $E[[U]]$. De plus il existe alors un $E[[U]]$-module de type fini r\'eflexif $M_1$ contenant $M/\mathrm{T}(M)$ tel que le quotient de $M_1$ par $M/\mathrm{T}(M)$ soit de longueur finie. De plus, si $[F: \Q_p] \leq 2$, on peut supposer que $M_1$ est un $E[[U]]$-module libre.
\end{cor}

\begin{proof}
On note $M^*=\hom_{E[[U]]}(M,E[[U]])$ le dual de $M$. On le munit d'une action de $\Gamma$ de la façon suivante. Si $\gamma \in \Gamma$, $f \in M^*$ et $x \in M$, on pose $(\gamma \cdot f)(x)=\gamma f(\gamma^{-1}x)$. L'application canonique de $ i : \, M \rightarrow M^{**}$ est alors $\Gamma$-\'equivariante. De plus, si $K$ d\'esigne le corps des fractions de $A$, $i \otimes_A K$ est un isomorphisme de $K$-espaces vectoriels de dimension finie. On en conclut que $\Ker (i)$ et $\Coker (i)$ sont des $E[[U]]$-modules de type fini, de torsion et stables par $\Gamma$. Ils sont donc de longueur finie d'apr\`es le corollaire \ref{cor-Ass}. Comme $M^{**}$ est sans torsion, on a en fait  $\Ker(i)=\mathrm{T}(M)$, ceci permet de conclure en prenant $M_1=M^{**}$.
\\
Si $\dim(E[[U]])=1$, $E[[U]]$ est un anneau de valuation discr\`ete, tout $E[[U]]$-module de type fini sans torsion est donc libre. D'apr\`es la proposition $2$ de \cite{Samuel}, tout module r\'eflexif de type fini sur un anneau local noeth\'erien r\'egulier de dimension $2$ est libre, donc $M^{**}$ est libre si $\dim(E[[U]])=2$.
\end{proof}

On peut se demander \`a quel point, lorsque $[F: \Q_p] > 2$, un $E[[U]]$-module r\'eflexif de type fini muni d'une action semi-lin\'eaire de $\Gamma$ est \'eloign\'e d'un $E[[U]]$-module libre. Nous ne connaissons pas d'exemple de tel module qui ne soit pas libre. Le mieux que l'on puisse dire sur un tel module est contenu dans la proposition suivante.

\begin{prop}
Soit $M$ un $E[[U]]$-module r\'eflexif de type fini muni d'une action semi-lin\'eaire de $\Gamma$. Pour tout id\'eal premier $\p \neq \m$ de $E[[U]]$, le $E[[U]]_{\p}$-module $M_{\p}$ est libre.
\end{prop}

\begin{proof}
Posons $A=E[[U]]$. Pour $i\geq 1$, notons $\otimes^i_AM:=M\otimes_{A} \cdots\otimes_{A} M$ le $A$-module produit tensoriel de $i$ copies de $M$. Ce sont des $A$-modules de type fini munis naturellement d'une action semi-lin\'eaire de $\Gamma$. Soit $\mathrm{T}(\otimes^i_AM)$ le sous-module de torsion de $\otimes^i_AM$; il est clairement stable par $\Gamma$. D'apr\`es le corollaire \ref{cor-Ass}, on voit que $\mathrm{Ass}\bigl(\mathrm{T}(\otimes^i_AM)\bigr)\subset\{\m\}$, d'o\`u $\bigl(\mathrm{T}(\otimes^i_AM)\bigr)_{\p}=0$ pour tout id\'eal premier $\p\neq \m$ par \cite[\S3, corollaire 1]{Bou}.
On d\'eduit alors des isomorphismes naturels (avec les notations \'evidentes)\[\bigl(\mathrm{T}(\otimes^i_AM)\bigr)_{\p}\cong \mathrm{T}\bigl((\otimes^i_AM\bigr)_{\p})\cong \mathrm{T}(\otimes^i_{A_{\p}}M_{\p})\] que $\otimes^i_{A_{\p}}M_{\p}$  est un $A_{\p}$-module sans torsion pour tout $i\geq 1$. Comme $A_{\p}$ est un anneau local r\'egulier non ramifi\'e  puisque $A$ l'est (car de caract\'eristique $p$, voir \cite[p.634]{Aus}), le th\'eor\`eme 3.2 de \cite{Aus} nous permet de conclure.
\end{proof}

\noindent
IRMAR - UMR CNRS 6625\\
Campus Beaulieu, 35042 Rennes cedex, France\\
{\it Adresse e-mail:} {\ttfamily yongquan.hu@univ-rennes1.fr}
\\

\noindent
Laboratoire de Math\'ematiques de Versailles\\
UMR CNRS 8100\\
45, avenue des \'Etats Unis - B\^atiment Fermat\\
F--78035 Versailles Cedex, France\\
{\it Adresse e-mail:} {\ttfamily Stefano.Morra@math.uvsq.fr}\\
{\it Adresse e-mail:} {\ttfamily benjamin.schraen@math.uvsq.fr}

\end{document}